\documentclass[USenglish]{article}

\usepackage[utf8]{inputenc}
\usepackage[small]{dgruyter}
\usepackage{microtype}
\usepackage{tikz}

\newtheorem{theorem}{Theorem}
\newtheorem{lemma}{Lemma}

\newtheorem{problem}{Problem}

\newtheorem{conjecture}{Conjecture}

\begin{document}

  \articletype{...}

  \author[1]{Yusuf Hafidh}
  \author[3]{Rizki Kurniawan}
  \author[1]{Suhadi Saputro}
  \author*[0]{Rinovia Simanjuntak}
  \author[3]{Steven Tanujaya}
  \author[1]{Saladin Uttunggadewa}
  \runningauthor{Simanjuntak \emph{et. al.}}
  \affil[0]{Combinatorial Mathematics Research Group, Faculty of Mathematics and Natural Sciences, Institut Teknologi Bandung}
  \affil[1]{Combinatorial Mathematics Research Group, Faculty of Mathematics and Natural Sciences, Institut Teknologi Bandung}
  \affil[3]{Bachelor Program in Mathematics, Faculty of Mathematics and Natural Sciences, Institut Teknologi Bandung}
  \title{Multiset Dimensions of Trees}
  \runningtitle{Multiset Dimensions of Trees}
  \subtitle{...}
  \abstract{
  Let $G$ be a connected graph and $W$ be a set of vertices of $G$. The representation multiset of a vertex $v$ with respect to $W$, $r_m (v|W)$, is defined as a multiset of distances between $v$ and the vertices in $W$. If $r_m (u |W) \neq r_m(v|W)$ for every pair of distinct vertices $u$ and $v$, then $W$ is called an m-resolving set of $G$. If $G$ has an m-resolving set, then the cardinality of a smallest m-resolving set is called the multiset dimension of $G$, denoted by $md(G)$; otherwise, we say that $md(G) = \infty$.

  In this paper, we show that for a tree $T$ of diameter at least 2, if $md(T) < \infty$, then $md(T) \leq n-2$. We conjecture that this bound is not sharp in general and propose a sharp upper bound. We shall also provide necessary and sufficient conditions for caterpillars and lobsters having finite multiset dimension. Our results partially settled a conjecture and an open problem proposed in \cite{SSV}.
  }

  \keywords{m-resolving set, multiset dimension, trees}
  \classification[PACS]{...}
  \communicated{...}
  \dedication{...}
  \received{...}
  \accepted{...}
  \journalname{...}
  \journalyear{...}
  \journalvolume{..}
  \journalissue{..}
  \startpage{1}
  \aop
  \DOI{...}

\maketitle

\section{Introduction}

Let $G$ be a simple and connected graph with vertex set $V(G)$. The \emph{distance $d(u,v)$ between two vertices $u,v \in V(G)$} is the length of a shortest path between them. The \emph{eccentricity of a vertex $v$}, $ecc(v)$ is the maximum distance from $v$ to other vertices in $G$. The \emph{radius of $G$} is $rad(G):=\min\{ecc(v):v\in G\}$. A \emph{center of $G$} is a vertex  with the smallest eccentricity (equal to the radius) and the set of centers of $G$ is denoted by $C(G)$. The diameter of $G$ is $diam(G):=\max\{ecc(v):v\in G\}$ and an \emph{end-vertex} is a vertex with the highest eccentricity (equal to the diameter).

The concept of multiset dimension was introduced in \cite{SSV} as a natural variation of metric dimension. In both concepts, the location of a vertex in a graph is uniquely identified by utilising the distance from that vertex to a set of "landmarks". Each vertex is then allocated a distinct "coordinate": in metric dimension, the coordinates are vectors, while in mustiset dimension, the coordinates are instead multisets.
	
For an ordered set of $k$ vertices $W = \{ w_1, w_2, \dots , w_k \}$, the \emph{representation of a vertex $v$ with respect to $W$} is the ordered $k$-tuple $$r(v|W) = ( d(v,w_1), d(v,w_2), \dots , d(v,w_z)).$$ $W$ is a \emph{resolving set of $G$} if every two vertices of $G$ have distinct representations. A resolving set with minimum cardinality is called a \emph{basis} and the number of vertices in a basis is called the \emph{metric dimension}, denoted by $dim(G)$.

Now suppose that $W'$ is an unordered subset of $V(G)$. The \emph{representation multiset of $v$ with respect to $W'$}, $r_m (v|W')$, is defined as a multiset of distances between $v$ and the vertices in $W'$. If $r_m (u|W') \neq r_m(v|W')$ for every pair of distinct vertices $u$ and $v$, then $W'$ is called an \emph{m-resolving set} of $G$. If $G$ has an m-resolving set, then the cardinality of a smallest m-resolving set is called the \emph{multiset dimension of $G$}, denoted by $md(G)$; otherwise, we say that $md(G) = \infty$.

In \cite{SSV}, a few basic results of multiset dimension were proved. Here we list those connected to the results of this paper.
\begin{lemma} \cite{SSV}
\begin{enumerate}
  \item The multiset dimension of a graph $G$ is one if and only if $G$ is a path. \label{path}
  \item No graph has multiset dimension $2$. \label{not2}
  \item Let $G$ be a graph other than a path. Then $md(G) \ge 3$. \label{bound1}
  \item If $G$ is a non-path graph of diameter at most $2$, then $md(G) = \infty$. \label{diam2}
\end{enumerate}
\end{lemma}

Another basic property needs the definition of twin vertices as follow. Two vertices $u$ and $v$ are said to be \emph{twins} if $N(u)\setminus\{v\}=N(v)\setminus\{u\}$.
We define a relation $\sim$ where $u \sim v$ if and only if $u=v$ or $u$ and $v$ are twins. It is quite obvious that $\sim$ is an equivalence relation on $V(G)$ and we denote by $v^*$ the equivalence class containing the vertex $v$. The fact that a pair of twins have the same distance to every other vertex gives rise to a necessary condition for graphs having finite multiset dimension.

\begin{lemma} \label{twins} \cite{SSV,KI18}
If $G$ contains a vertex $v$ with $|v^*|\geq 3$, then $md(G) = \infty$.
\end{lemma}

In general, the necessary conditions for the finiteness of multiset dimension in Lemmas \ref{diam2} and \ref{twins} are not sufficient, so it is interesting to find such conditions for trees, as proposed in \cite{SSV}.
\begin{problem} \cite{SSV} \label{chartree}
Characterize all trees having finite multiset dimension. Give exact values of the multiset dimension of trees if it is finite.
\end{problem}


It is obvious that if a graph $G$ on $n$ vertices has finite multiset dimension, then $md(G)\leq n$.
However it was proposed that the natural upper bound is not sharp as stated in the following conjecture.
\begin{conjecture} \cite{SSV} \label{n-1}
If $G$ is a graph on $n$ vertices having finite multiset dimension, then $md(G)\leq n-1$.
\end{conjecture}

In this paper, we show that Conjecture \ref{n-1} is true for trees in Section \ref{s_bounds}. Additionally, we provide necessary and sufficient conditions for caterpillars and lobsters having finite multiset dimension in Section \ref{s_catlob} which partially settled Open Problem \ref{chartree}.

To prove the aforementioned results, we shall utilise the following properties of trees, which follow from the fact that in a tree, there exists a unique path connecting every pair of distinct vertices.
\begin{lemma}
Let $T$ be a tree.
\begin{enumerate}
\item $rad(T)=\left\lceil \frac{diam(T)}{2} \right\rceil$.
\item If $T$ has even diameter then $|C(T)|=1$ and if $T$ has odd diameter then $C(T)=\{u,v\}$, where $u,v$ are two adjacent vertices.
\item For every $v\in V(T)$, $ecc(v)=rad(T)+d(v,C(T))$. \label{ecc}
\end{enumerate}
\end{lemma}

From now on, we shall denote by $P_n$ the path on $n$ vertices and $S_n$ the star on $n$ vertices.

\section{Upper bound for multiset dimensions of trees} \label{s_bounds}

In this section, we provide partial proof for Conjecture \ref{n-1}, that is the conjecture is true for trees. We start with trees of small diameter.

\begin{lemma} \label{diam3}
If $T$ is a tree of order $n$ and diameter 3 then $md(T)\leq n-2$.
\end{lemma}
\begin{proof}
Let $u$ and $v$ be the centers of $T$, by Lemma \ref{twins}, $deg(u)\leq 3$ and $deg(v)\leq 3$. Since $diam(T)=3$, the centers are not leaves, and so $2 \leq deg(u), deg(v) \leq 3$.

We shall consider 3 cases:
\begin{description}
  \item[\textbf{(1) $deg(u)=deg(v)=2$}:] Thus $T\approx P_4$ and $md(T)=1=n-3$.
  \item[\textbf{(2) $deg(u)=2$ and $deg(v)=3$:}] Let $N(u)=\{v,u_1\}$ and $N(v)=\{u,v_1,v_2\}$. Thus we have $R=\{u_1,u,v_1\}$ is an m-resolving set for $T$, which means $md(T)=3=n-2$.
  \item[\textbf{(3) $deg(u)=deg(v)=3$:}] Let $N(u)=\{v,u_1,u_2\}$ and $N(v)=\{u,v_1,v_2\}$. Therefore $R=\{u_1,u,v_1\}$ is an m-resolving set for $T$, or $md(T)=3=n-3$.
\end{description}
\end{proof}

\begin{theorem}\label{bounds}
Let $T$ be a tree of order $n$ and diameter at least $2$. If $md(T) < \infty$, then $md(T) \leq n-2$.
\end{theorem}
\begin{proof}
Let $T$ be a tree with finite multiset dimention. If $diam(T)=2$, by Lemma \ref{path} and \ref{diam2}, $T\approx P_3$ and the result follows.

Now let $diam(T)\geq 3$. First we prove $md(T)\leq n-1$. Let $R$ be arbitrary m-resolving set of $T$. Since if $R \ne V(T)$ we already have the desired result, consider $R=V(T)$. We claim that $R'=V(T)-C(T)$ is also an m-resolving set for $T$. Consider two cases based on the parity of the diameter of $T$.

\textbf{Case 1: $\mathbf{diam(T)}$ is even.} Recall that there exists a unique center. To the contrary, suppose that $R'$ is not an m-resolving set, then there exist two vertices $u$ and $v$ where $r_m(u|R')=r_m(v|R')$. The maximum element in $r_m(u|R')$ is $ecc_T(u)$, thus $ecc_T(u)=ecc_T(v)$. By (\ref{ecc}), $d(u,C(T))=ecc_T(u)-rad(T)=ecc_T(v)-rad(T)=d(v,C(T))$, and so $r_m(u|R)=r_m(u|R')\cup \{d(u,C(T))\} = r_m(v|R') \cup \{d(v,C(T))\} = r_m(v|R)$, which contradict the fact that $V(T)$ is a resolving set.
			
\textbf{Case 2: $\mathbf{diam(T)}$ is odd.} Let $C(T)=\{c_1,c_2\}$, where $c_1$ and $c_2$ are adjacent. Similar to the first case, it can be showed that $r_m(u|R)=r_m(u|R')\cup\{d(u,c_1),d(u,c_2)\}=r_m(u|R')\cup\{ecc_T(u)-rad(T),ecc_T(u)-rad(T)+1\}$, a contradiction.
			
Now we are ready to prove that $md(T)\leq n-2$. If $diam(T)=3$, the result follows from Lemma \ref{diam3}. Let $diam(T)\geq 4$ and let $R$ be an m-resolving set for $T$ with $|R| \leq n-1$. It is only necessary to consider the case when $R=V(T)-\{x\}$, since otherwise the desired result holds. Let $T'$ be the minimum induced subgraph of $T$ containing $R$, that is $T'=T-x$ if $x$ is a leaf and $T'=T$ otherwise. Note that $diam(T') \geq diam(T)-1 \geq 3$. We shall consider two cases, based on whether $x$ is a center of $T'$.
			
\textbf{Case 1: $\mathbf{x \notin C(T')}$.} We claim that $R'=R-C(T')=V(T)-\{x\}-C(T')$ is also an m-resolving set. By the construction of $T'$, the maximum element in $r_m(u|R')$ is $ecc_{T'}(u)$, for any $u\in V(T')$. Since $x\notin R'$ then $0\notin r_m(x|R')$, and $x\notin C(T')$ implies that $\max(r_m(x|R'))=rad(T')+d(x,C(T'))>rad(T')$. The only vertex $u$ other than $x$ with $0\notin r_m(u|R')$ is $u\in C(T')$, but if $u\in C(T')$ then $\max(r_m(x|R'))=rad(T')+d(x,C(T'))>rad(T')=ecc_{T'}(u)=\max(r_m(u|R'))$. Therefore $r_m(x|R') \ne r_m(u|R')$ for all $u\in V(T)-\{x\}$. Let $u,v$ be two vertices in $V(T)-\{x\}$. If $r_m(u|R')=r_m(v|R')$, then $r_m(u|R)=r_m(v|R)$, a contradiction. These show that the multiset representation of every vertex is distinct.

\textbf{Case 2 : $\mathbf{x\in C(T')}$.} Here $x$ is not a leaf and $T'=T$. Since $diam(T)\geq 4$, then neither a center nor a neighbor of a center is an end-vertex. We again separate our observation into two subcases, based on the diameter of $T$.
\begin{description}
\item[\textbf{Case 2.1:}] \textbf{$\mathbf{diam(T)}$ is even.} In this case $C(T)=\{x\}$. Let $N=N(x)$ be the set of neighbors of the center or the set of vertices with eccentricity $rad(T)+1$. Let $R'=R-N=V(T)-\{x\}-N$, we claim that $R'$ is also an m-resolving set. Since $diam(T)\geq 4$ the vertices in $N$ are not end-vertices, and so the maximum element in $r_m(u|R')$ is still $ecc_T(u)$. Since $x\notin R'$ and $x\in C(T)$, then $0\notin r_m(x|R')$ and $\max(r_m(x|R'))=ecc_T(x)=rad(T)$. If $u$ is a vertex other than $x$ with $0\notin r_m(u|R')$, then $u\in N$ and $\max(r_m(x|R')) = rad(T) \ne rad(T) + 1 = \max(r_m(u|R'))$. Therefore $r_m(x|R')\ne r_m(u|R')$ for all $u\in V(T)-\{x\}$.
			
Now let $k=|N|$ and consider a vertex $v\in V(T)-\{x\}$. Thus,
			\begin{align*}
			r_m(v|R) &=r_m(v|R') \cup \{d(v,y):y\in N\}\\
			        &=r_m(v|R') \cup \{d(v,x)-1,(d(v,x)+1)^{k-1}\}\\
			        &=r_m(v|R') \cup \{ecc_T(v)-rad(T)-1,(ecc_T(v)-rad(T)+1)^{k-1}\}.
			\end{align*}
Since the maximum element in $r_m(u|R')$ is $ecc_{T}(u)$, if $r_m(u|R')=r_m(v|R')$, we will obtain $r_m(u|R)=r_m(v|R)$, a contradiction.
\item[\textbf{Case 2.2:}] \textbf{$\mathbf{diam(T)}$ is odd.} Since $|C(T)|=2$, let $C(T)=\{x,y\}$.

First we consider the case when $|T|$ is odd. Obviously, $|T-C(T)|$ is also odd. Consider $R'=R-\{y\}=V(T)-\{x,y\}$ and $T_x$ and $T_y$ are the components of $T-xy$ containing $x$ and $y$, respectively. Let $u$ and $v$ be two vertices in $T$, with $r_m(u|R')=r_m(v|R')$. This means $ecc(u)=ecc(v)$. If both $u$ and $v$ is in either $T_x$ or $T_y$, then they have the same distance to $y$ and $r_m(u|R)=r_m(v|R)$, a contradiction. If $u$ in $T_x$ and $v$ in $T_y$, then $d(u,x)=d(u,C(T))=ecc(u)-rad(T)=ecc(v)-rad(T)=d(v,C(T))=d(v,y)$, and so $d(u,v)$ is odd. Now we count the number of vertices in $R'$ with odd and even distance to $u$, and denote them with $o$ and $\epsilon$, respectively. Since $o + \epsilon = |R'|$ is odd, then $o \ne \epsilon$. However $d(u,v)$ is odd, and so, by the uniqueness of path between two vertices, a vertex with odd distance to $u$ has even distance to $v$ and vice versa. This means the number of vertices in $R'$ with odd distance to $v$ is $\epsilon$ which is not equal to the number of vertices with odd distance to $u$, a contradiction.

Now we consider the case when $|T|$ is even. Our idea is to construct a new m-resolving set of odd cardinality by removing some neighbors of the center. By doing so, two vertices in $T_x$ (or $T_y$) will have the same distance to the removed vertices; while a vertex in $T_x$ and a vertex in $T_y$ will have different number of vertices in $R'$ with odd distance. As we already established in the previous cases, this will guarantee that the newly constructed set is m-resolving. Our construction depends on the degrees of $x$ and $y$. If either $deg(x)$ or $deg(y)$ is odd, then either $|N|=|N(x)-\{y\}|$ or $|N|=|N(y)-\{x\}|$ is even. Choose the $N$ with even cardinality. Since $|T|$ is even then $|R|$ is odd and $|R'|=|R-N|$ is also odd. If both $deg(x)$ and $deg(y)$ is even, let $N_1=N(x)-\{y\}$ and $N_2=N(y)-\{x\}$. Thus $R'=V(T)-N_1-N_2$ is the required new m-resolving set.
\end{description}
\end{proof}
	
To study whether the upper bound in Theorem \ref{bounds} is sharp, we conducted exhaustive search for multiset basis for all trees of order up to 10, which were generated by using McKay's geng software \cite{MP13}. We found that there are only two trees of order $n$ with multiset dimension $n-2$. They are the path on 3 vertices and the tree on 5 vertices constructed from the star on 4 vertices by subdividing exactly one of its edges. The complete statistics for trees of order $n$, $6 \leq n \leq 10$ can be seen in Table \ref{trees}.

\begin{table}[h]
\begin{center}
\begin{tabular}{c|c|c|c|c|c|c|c|c|c|c}
  \tiny{$n$} & \tiny{\# trees} & \tiny{$md=\infty$} & \tiny{$md=1$} & \tiny{$md=3$} & \tiny{$md=4$} & \tiny{$md=5$} & \tiny{$md=6$} & \tiny{$md=7$} & \tiny{$md=8$} & \tiny{$md=9$} \\
  \hline
  \tiny{6} & \tiny{6} & \tiny{2} & \tiny{1} & \tiny{3} & \tiny{0} & \tiny{0} & \tiny{0} & \tiny{-} & \tiny{-} & \tiny{-} \\
  \tiny{7} & \tiny{11} & \tiny{4} & \tiny{1} & \tiny{5} & \tiny{1} & \tiny{0} & \tiny{0} & \tiny{0} & \tiny{-} & \tiny{-} \\
  \tiny{8} & \tiny{23} & \tiny{9} & \tiny{1} & \tiny{11} & \tiny{2} & \tiny{0} & \tiny{0} & \tiny{0} & \tiny{0} & \tiny{-} \\
  \tiny{9} & \tiny{47} & \tiny{20} & \tiny{1} & \tiny{23} & \tiny{3} & \tiny{0} & \tiny{0} & \tiny{0} & \tiny{0} & \tiny{0} \\
  \tiny{10} & \tiny{106} & \tiny{48} & \tiny{1} & \tiny{53} & \tiny{2} & \tiny{2} & \tiny{0} & \tiny{0} & \tiny{0} & \tiny{0} \\
\end{tabular}
\caption{Multiset dimensions of trees of order up to 10.} \label{trees}
\end{center}
\end{table}

Based on the result of our computer search, we would like to propose the following.
\begin{conjecture}\label{upperbound}
Let $T$ be a tree. If $md(T) < \infty$, then $md(T) \leq n-diam(T)+1$ and the bound is sharp.
\end{conjecture}

If Conjecture \ref{upperbound} is true, an example for the sharpness of the bound is the tree on $n$ vertices constructed from the star on 4 vertices by subdividing exactly one of its edges $n-4$ times. This tree has multiset dimension 3 which is equal to $n-(n-2)+1$, where $n-2$ is the diameter of the tree.


\section{Caterpillars and lobsters with finite multiset dimensions} \label{s_catlob}

In this paper we define caterpillars and lobsters by using the notion of a $k$-center-path. Let $G$ be a connected graph, a subgraph $P$ of $G$ is called a \emph{$k$-center-path of $G$} if $P$ is a path and $d(u,P)\leq k$ for every vertex $u$ in $G$. A \emph{minimum $k$-center-path} is a $k$-center path with minimum length. A \emph{caterpillar} is a tree containing a $1$-center-path and a \emph{lobster} is a tree containing a $2$-center-path. Let $T$ be a rooted tree with $v$ as its root, a \emph{separation of $T$}, denoted by $[T]$, is a graph obtained by subdividing all edge attached to $v$ and then deleting $v$. Note that the number of components in $[H]$ is the degree of $v$ in $H$.

We start by characterising all lobsters having finite multiset dimension.
\begin{theorem}\label{Lobster}
Let $G$ be a lobster. If $P$ is the minimum $2$-center path of $G$, then the following are equivalent.
\begin{itemize}
\item[(1)] $G$ has finite multiset dimension.
\item[(2)] The only component of $G-E(P)$ with infinite multiset dimension is an $S_4$.
\item[(3)] If $H$ is a component of $G-E(P)$ then $[H]$ has at most $4$ components which are either a $P_2$, a $P_3$, or an $S_4$, with at most two $P_2$s and two $S_4$s.
\end{itemize}
\end{theorem}
\begin{proof}
We will prove $(2) \Rightarrow (3) \Rightarrow (1) \Rightarrow (2)$.
\begin{itemize}
\item[\textbf{$(2)\Rightarrow (3)$} ] Let $H$ be a component of $G-E(P)$ and $v$ the root of $H$ in $P$. If $H$ has infinite multiset dimension then $H$ is an $S_4$ which fullfils (3). Now consider that $H$ has finite multiset dimension.
			
Let $u$ be a neighbor of $v$ in $H$. Since $G$ is a lobster then all neighbors of $u$ other than $v$ are leaves, and by Lemma \ref{twins}, $deg(u)\leq 3$. This means the component in $[H]$ containing $u$ is either a $P_2$, a $P_3$, or an $S_4$.
			
Let $R$ be an m-resolving set of $H$, if $[H]$ contains an $S_4$ then exactly one of the leaf of $S_4$ is in $R$. If there are more than two $S_4$s in $[H]$, then there exist two $S_4$s with either both centers are in $R$ or not in $R$, which results on both centers having the same multiset representations with respect to $R$, a contradiction. The fact that $[H]$ can not have more than two $P_2$s is due to Lemma \ref{twins}.
			
Thus we are left to prove that $[H]$ has at most $4$ components. Since for each component of $[H]$, there are at most one vertex in $R$ with distance two from a vertex $v$ in $H$, we shall catagorize each component of $[H]$ into the following 4 types:
\begin{itemize}
\item \emph{type 0}: The component has no vertex in $R$,
\item \emph{type 1}: The component has one vertex in $R$ with distance 1 to $v$,
\item \emph{type 2}: The component has one vertex in $R$ with distance 2 to $v$, or
\item \emph{type 12}: The component has two vertices in $R$, each with distance 1 or 2 to $v$.
\end{itemize}
			
Suppose that $[H]$ has more than $4$ components. Thus there are two components with the same type, and so the two neighbors of $v$ in those two components will have the same multiset representation with respect to $R$, a contradiction.
			
\item[\textbf{$(3)\Rightarrow (1)$} ] Let $G$ be a graph that satisfies (3), $P=p_0,p_1,\ldots,p_n$, $H_i$ be the component of $G-E(P)$ containing $p_i$, and $d_i$ be the degree of $p_i$ in $H_i$. For $i=0,1,\ldots, n$, let $R(H_i)$ be a the resolving set of $H_i$ with maximum cardinality which does not contain the root.
			
We shall construct an m-resolving set for $G$. If $d_i=0$, define $R(H_i)=\emptyset$, otherwise do the folowing steps.
\begin{enumerate}
\item Define $H_i^{(1)},H_i^{(2)},\ldots, H_i^{(d_i)}$ a non-increasing sequence of components of $H_i$ with component ordering: $S_4 > P_3 > P_2$.
\item If $H_i^{(1)} \not\approx P_2$, choose exactly two vertices in $H_i^{(1)}$, each with distance $1$ and $2$ to $p_i$, as members of $R(H_i)$. If $H_i^{(1)} \approx P_2$, define $H_i^{(d_i+1)}\approx P_2$.
\item If $H_i^{(2)}$ exists and is not a $P_2$, choose exactly one vertex in $H_i^{(2)}$ with distance $2$ to $p_i$ as a member of $R(H_i)$. If $H_i^{(2)} \approx P_2$, define $H_i^{(d_i+e_i)}\approx P_2$ with $e_i=1$ if $H_i^{(1)} \not\approx P_2$ and $e_i=2$ otherwise.
\item If $H_i^{(3)}$ exists, choose the vertex in $H_i^{(3)}$ with distance $1$ to $p_i$ as a member of $R(H_i)$.
\end{enumerate}
Since $P$ is a minimum path and $R(H_i)$ a maximum resolving set, for $i=0$ and $i=n$, there exists a vertex in $R(H_i)$ with distance two to $P$. This means the maximum element in $r_m\left(v|\cup_{i=0}^n R(H_i)\right)$ is $ecc(v)$.	

Now we shall propose two constructions of an m-resolving set for $G$, depending on the diameter of $G$.
\begin{description}
\item[\textbf{Construction 1}: for odd $diam(G)$. ] If $\sum_{i=0}^{n} |R(H_i)|$ is odd, define $R=\cup_{i=0}^n R(H_i)$, otherwise  $R=p_0\cup\left(\cup_{i=0}^n R(H_i)\right)$. To prove that $R$ is a resolving set, we only need to show that vertices with the same eccentricity have different representations. In Figure \ref{diamodd}, the boxed vertices have the same eccentricity. Here we define sides as components of $G-c_1c_2$, where $c_1$ and $c_2$ are the centers.

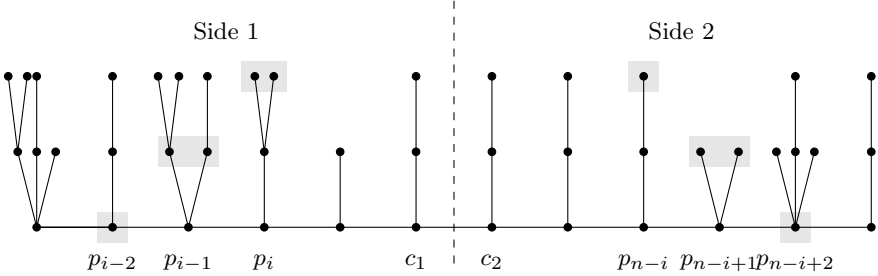
\begin{figure} [h]
\begin{center}
				\begin{tikzpicture}[scale=0.5]
				\fill [opacity=0.1] (-5.6,4.4) -- (-4.4,4.4) -- (-4.4,3.6) -- (-5.6,3.6) -- cycle;
				\fill [opacity=0.1] (-7.8,2.4) -- (-6.2,2.4) -- (-6.2,1.6) -- (-7.8,1.6) -- cycle;
				\fill [opacity=0.1] (-9.4,0.4) -- (-8.6,0.4) -- (-8.6,-0.4) -- (-9.4,-0.4) -- cycle;
				\fill [opacity=0.1] (8.6,0.4) -- (9.4,0.4) -- (9.4,-0.4) -- (8.6,-0.4) -- cycle;
				\fill [opacity=0.1] (6.2,2.4) -- (7.8,2.4) -- (7.8,1.6) -- (6.2,1.6) -- cycle;
				\fill [opacity=0.1] (4.6,4.4) -- (5.4,4.4) -- (5.4,3.6) -- (4.6,3.6) -- cycle;
				\draw (-9.0,0.0) -- (-9.0,2.0) -- (-9.0,4.0) (-11.0,0.0)-- (-10.5,2.0) (-11.0,0.0)-- (-11.0,2.0) (-11.0,0.0)-- (-11.5,2.0)	(-11.0,0.0)-- (-9.0,0.0)	(-11.0,2.0)-- (-11.0,4.0)	(-11.5,2.0)-- (-11.25,4.0)	(-11.5,2.0)-- (-11.75,4.0)	(-7.0,0.0)-- (-7.5,2.0)	(-7.0,0.0)-- (-6.5,2.0)	(-6.5,2.0)-- (-6.5,4.0)	(-7.5,2.0)-- (-7.25,4.0)	(-7.5,2.0)-- (-7.8,4.0) (-5.0,0.0)-- (-5.0,2.0)	(-5.0,2.0)-- (-5.25,4.0)	(-5.0,2.0)-- (-4.75,4.0)	(-3.0,0.0)-- (-3.0,2.0)	(-1.0,0.0)-- (-1.0,2.0)	(-1.0,2.0)-- (-1.0,4.0)	(1.0,0.0)-- (1.0,2.0)	(1.0,2.0)-- (1.0,4.0)	(3.0,0.0)-- (3.0,2.0) (3.0,2.0)-- (3.0,4.0)	(5.0,0.0)-- (5.0,2.0)	(5.0,2.0)-- (5.0,4.0)	(7.5,2.0)-- (7.0,0.0)			(9.0,0.0)-- (9.5,2.0)	(11.0,0.0)-- (11.0,2.0)	(11.0,2.0)-- (11.0,4.0)	(9.0,0.0)-- (8.5,2.0)	(9.0,0.0)-- (9.0,2.0)	(9.0,2.0)-- (9.0,4.0)	(7.0,0.0)-- (6.5,2.0)	(-11.0,0.0)-- (11.0,0.0);
				\draw  [fill] (-1.0,0.0) circle (0.1)	(1.0,0.0) circle (0.1)	(3.0,0.0) circle (0.1)	(5.0,0.0) circle (0.1)	(7.0,0.0) circle (0.1)	(-3.0,0.0) circle (0.1)	(-5.0,0.0) circle (0.1)	(-7.0,0.0) circle (0.1)	(-9.0,0.0) circle (0.1)	(9.0,0.0) circle (0.1)	(11.0,0.0) circle (0.1)	(-11.0,0.0) circle (0.1)	(-11.5,2.0) circle (0.1)	(-11.0,2.0) circle (0.1)	(-10.5,2.0) circle (0.1)	(-11.0,4.0) circle (0.1)	(-11.25,4.0) circle (0.1)	(-11.75,4.0) circle (0.1)	(-9.0,2.0) circle (0.1)	(-9.0,4.0) circle (0.1)	(-7.5,2.0) circle (0.1)	(-6.5,2.0) circle (0.1)	(-6.5,4.0) circle (0.1)	(-7.25,4.0) circle (0.1)	(-7.8,4.0) circle (0.1)	(-5.0,2.0) circle (0.1)	(-5.25,4.0) circle (0.1)	(-4.75,4.0) circle (0.1)	(-3.0,2.0) circle (0.1)	(-1.0,2.0) circle (0.1)	(-1.0,4.0) circle (0.1)	(1.0,2.0) circle (0.1)	(1.0,4.0) circle (0.1)	(3.0,2.0) circle (0.1) (3.0,4.0) circle (0.1)	(5.0,2.0) circle (0.1)	(5.0,4.0) circle (0.1)	(7.5,2.0) circle (0.1) (9.5,2.0) circle (0.1)	(11.0,2.0) circle (0.1)	(11.0,4.0) circle (0.1)	(8.5,2.0) circle (0.1)	(9.0,2.0) circle (0.1)	(9.0,4.0) circle (0.1)	(6.5,2.0) circle (0.1);
				\draw [dashed] (0,6)--(0,-1);
				\draw (-6,5.2)node{Side $1$};
				\draw (6,5.2)node{Side $2$};
				\draw (-9,-0.5)node[below]{$p_{i-2}$};
				\draw (-7,-0.5)node[below]{$p_{i-1}$};
				\draw (-5,-0.5)node[below]{$p_{i}$};
				\draw (-1,-0.5)node[below]{$c_1$};
				\draw (1,-0.5)node[below]{$c_2$};
				\draw (9,-0.5)node[below]{$p_{n-i+2}$};
				\draw (7,-0.5)node[below]{$p_{n-i+1}$};
				\draw (5,-0.5)node[below]{$p_{n-i}$};
				\end{tikzpicture}
			\end{center}
\caption{A lobster $G$ with odd diameter.} \label{diamodd}
\end{figure}

The vertices in the same box will have different representation with respect to $R$, because $R(H_i)$ is a resolving set for $H_i$. Now consider vertices in different boxes. Vertices in a box in Side 1 are with the same distance to vertices in $M:=\cup_{k=i+1}^n R(H_k)$, but they have distinct multiset representations with respect to $N:=R-M$, since the maximum distance to vertices in $N$ is distinct. This means that their multiset representations with respect to $R$ is also distinct. Similar argument can be applied to vertices in a box in Side 2.
			
Now let $u$ be a vertex in a box in Side $1$ and $v$ be a vertex in a box in Side $2$. Since $diam(G)$ is odd, $d(u,v)$ is also odd. Let $o$ and $\epsilon$ be the number of vertices in $R$ with odd and even distances to $u$, respectively. Since $o + \epsilon = |R|$ is odd, then $o \ne \epsilon$. Since $d(u,v)$ is odd, a vertex with odd distance to $u$ will have even distance to $v$ and vice versa. This means the number of vertices in $R$ with odd distance to $v$ is $\epsilon$ which is not equal to the number of vertices with odd distance to $u$. Therefore $r_m(u|R) \ne r_m(v|R)$.
			
\item[\textbf{Construction 2}: for even $diam(G)$. ] Here $H_{n/2}$ will be exactly in the middle and thus will not be included in any side. For $i=0,1,2$ we define $$a_i:=|\{w\in \cup_{i=0}^n R(H_i)|w\ \text{in Side 1 and }ecc(w)=diam(G)-i\}|$$ and $$b_i:=|\{w\in \cup_{i=0}^n R(H_i)|w\ \text{in Side 2 and }ecc(w)=diam(G)-i\}|.$$ By considering the vectors $(a_0,a_1,a_2)$ and $(b_0,b_1,b_2)$; the side with larger vector (lexicographically) will be named the dominant side, and if $(a_0,a_1,a_2)=(b_0,b_1,b_2)$ we name Side $1$ as the dominant side.
			
Assume that $n>2$. If Side $1$ is dominant, define $R=\{p_0,p_{n/2+1}\}\cup_{i=0}^n R(H_i)$, otherwise define $R=\{p_n,p_{n/2-1}\}\cup_{i=0}^n R(H_i)$. The proof that $R$ is an m-resolving set will only be given for the case when Side $1$ is dominant, the other case can be proved similarly.

\begin{figure}[h]		
			\begin{center}
				\begin{tikzpicture}[scale=0.6]
				\fill [opacity=0.1] (-5.6,4.4) -- (-4.4,4.4) -- (-4.4,3.6) -- (-5.6,3.6) -- cycle;
				\fill [opacity=0.1] (-7.8,2.4) -- (-6.2,2.4) -- (-6.2,1.6) -- (-7.8,1.6) -- cycle;
				\fill [opacity=0.1] (-9.4,0.4) -- (-8.6,0.4) -- (-8.6,-0.4) -- (-9.4,-0.4) -- cycle;
				\fill [opacity=0.1] (6.6,0.4) -- (7.4,0.4) -- (7.4,-0.4) -- (6.6,-0.4) -- cycle;
				\fill [opacity=0.1] (4.6,2.4) -- (5.4,2.4) -- (5.4,1.6) -- (4.6,1.6) -- cycle;
				\fill [opacity=0.1] (2.6,4.4) -- (3.4,4.4) -- (3.4,3.6) -- (2.6,3.6) -- cycle;
				\draw (-9.0,0.0) -- (-9.0,2.0) -- (-9.0,4.0) (-11.0,0.0)-- (-10.5,2.0) (-11.0,0.0)-- (-11.0,2.0) (-11.0,0.0)-- (-11.5,2.0)	(-11.0,0.0)-- (-9.0,0.0)	(-11.0,2.0)-- (-11.0,4.0)	(-11.5,2.0)-- (-11.25,4.0)	(-11.5,2.0)-- (-11.75,4.0)	(-7.0,0.0)-- (-7.5,2.0)	(-7.0,0.0)-- (-6.5,2.0)	(-6.5,2.0)-- (-6.5,4.0)	(-7.5,2.0)-- (-7.25,4.0)	(-7.5,2.0)-- (-7.8,4.0) (-5.0,0.0)-- (-5.0,2.0)	(-5.0,2.0)-- (-5.25,4.0)	(-5.0,2.0)-- (-4.75,4.0)	(-3.0,0.0)-- (-3.0,2.0)	(-1.0,0.0)-- (-1.0,2.0)	(-1.0,2.0)-- (-1.0,4.0)	(1.0,0.0)-- (1.0,2.0)	(1.0,2.0)-- (1.0,4.0)	(3.0,0.0)-- (3.0,2.0) (3.0,2.0)-- (3.0,4.0)	(5.0,0.0)-- (5.0,2.0)	(5.0,2.0)-- (5.0,4.0)	(7.5,2.0)-- (7.0,0.0)	(9.0,0.0)-- (9.5,2.0)	(9.0,0.0)-- (8.5,2.0)	(9.0,0.0)-- (9.0,2.0)	(9.0,2.0)-- (9.0,4.0)	(7.0,0.0)-- (6.5,2.0)	(-11.0,0.0)-- (9.0,0.0);
				\draw  [fill] (-1.0,0.0) circle (0.1)	(1.0,0.0) circle (0.1)	(3.0,0.0) circle (0.1)	(5.0,0.0) circle (0.1)	(7.0,0.0) circle (0.1)	(-3.0,0.0) circle (0.1)	(-5.0,0.0) circle (0.1)	(-7.0,0.0) circle (0.1)	(-9.0,0.0) circle (0.1)	(9.0,0.0) circle (0.1)	(-11.0,0.0) circle (0.1)	(-11.5,2.0) circle (0.1)	(-11.0,2.0) circle (0.1)	(-10.5,2.0) circle (0.1)	(-11.0,4.0) circle (0.1)	(-11.25,4.0) circle (0.1)	(-11.75,4.0) circle (0.1)	(-9.0,2.0) circle (0.1)	(-9.0,4.0) circle (0.1)	(-7.5,2.0) circle (0.1)	(-6.5,2.0) circle (0.1)	(-6.5,4.0) circle (0.1)	(-7.25,4.0) circle (0.1)	(-7.8,4.0) circle (0.1)	(-5.0,2.0) circle (0.1)	(-5.25,4.0) circle (0.1)	(-4.75,4.0) circle (0.1)	(-3.0,2.0) circle (0.1)	(-1.0,2.0) circle (0.1)	(-1.0,4.0) circle (0.1)	(1.0,2.0) circle (0.1)	(1.0,4.0) circle (0.1)	(3.0,2.0) circle (0.1)	(3.0,4.0) circle (0.1)	(5.0,2.0) circle (0.1)	(5.0,4.0) circle (0.1)	(7.5,2.0) circle (0.1)	(9.5,2.0) circle (0.1)	(8.5,2.0) circle (0.1)	(9.0,2.0) circle (0.1)	(9.0,4.0) circle (0.1)	(6.5,2.0) circle (0.1);
				\draw [dashed] (0,6)--(0,-1);
				\draw [dashed] (-2,6)--(-2,-1);
				\draw (-6,5.2)node{Side $1$};
				\draw (4,5.2)node{Side $2$};
				\draw (-11,-0.5)node[below]{$p_0$};
				\draw (-1,-0.5)node[below]{$p_{n/2}$};
				\draw (1,-0.5)node[below]{$p_{n/2+1}$};
				\draw (9,-0.5)node[below]{$p_{n}$};
				\end{tikzpicture}
			\end{center}
\caption{A lobster $G$ with even diameter.} \label{diameven}
\end{figure}
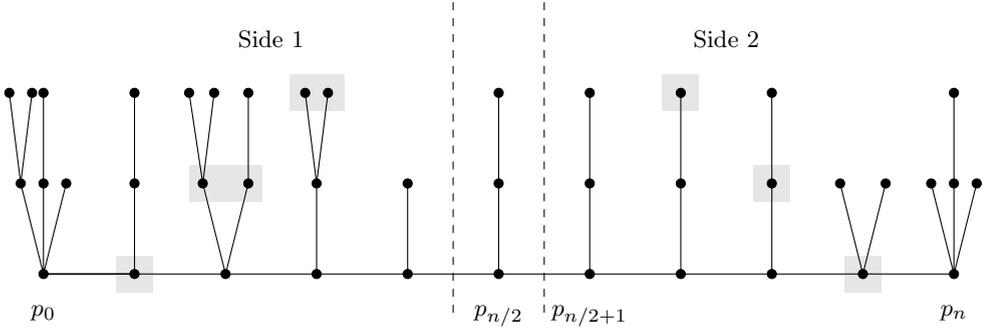
			
Similar argument from Construction 1 can be applied to prove that the vertices in the same side have distinct representations. Let $u$ and $v$ be vertices with the same eccentricity in Side $1$ and side $2$, respectively. For $i=0,1,\ldots,ecc(v)$, we define $m_v(i)$ as the multiplicity of $i$ in $r_m(v|R)$. If $v$ is in $H_{n/2+1}$, then $m_v:=(m_v(ecc(v)),m_v(ecc(v)-1),m_v(ecc(v)-2))=(a_0,a_1,a_2+b_0+1)$, otherwise $m_v=(a_0,a_1,a_2+1)$. If $u$ is in $H_{n/2-1}$, then $m_u=(b_0,b_1,b_2+a_0)$, otherwise $m_u=(b_0,b_1,b_2)$.
			
In cases other than $u\in V(H_{n/2-1})$ and $v\notin V(H_{n/2}+1)$, we have $m_v > m_u$, and so $r_m(u|R)\ne r_m(v|R)$. Therefore we only need to prove the case when $u \in H_{n/2-1}$ and $v\notin H_{n/2+1}$. Consider the lobster in Figure \ref{uv}, where the black vertices indicate the members of $R$.
\begin{figure} [h]
			\begin{center}
				\begin{tikzpicture}[scale=0.6]
				\draw [fill,opacity=0.1] (-3.5,4.4)--(-2.5,4.4)--(-2.5,3.6)--(-3.5,3.6)--cycle;
				\draw [fill,opacity=0.1] (3.25,2.4)--(4.25,2.4)--(4.25,1.6)--(3.25,1.6)--cycle;
				\draw (-5.0,0.0)-- (-5.0,2.0);
				\draw (-3.0,0.0)-- (-3.5,2.0);
				\draw (-3.0,0.0)-- (-2.5,2.0);
				\draw (-1.0,0.0)-- (-1.0,2.0);
				\draw (-1.0,2.0)-- (-1.0,4.0);
				\draw (1.0,0.0)-- (1.0,2.0);
				\draw (1.0,2.0)-- (1.0,4.0);
				\draw (3.0,0.0)-- (3.5,2.0);
				\draw (3.5,2.0)-- (3.5,4.0);
				\draw (-3.5,2.0)-- (-3.25,4.0);
				\draw (-3.5,2.0)-- (-3.75,4.0);
				\draw (-2.5,2.0)-- (-2.25,4.0);
				\draw (-2.5,2.0)-- (-2.75,4.0);
				\draw (2.5,2.0)-- (2.75,4.0);
				\draw (2.5,2.0)-- (2.25,4.0);
				\draw (2.5,2.0)-- (3.0,0.0);
				\draw (3.0,0.0)-- (3.0,2.0);
				\draw (3.0,0.0)-- (4.0,2.0);
				\draw (-9.0,0.0)-- (7.0,0.0);
				\draw (5.0,0.0)-- (5.0,2.0);
				\draw (5.0,2.0)-- (5.0,4.0);
				\draw (5.0,0.0)-- (5.5,2.0);
				\draw (5.0,0.0)-- (4.5,2.0);
				\draw [fill=white] (-1.0,0.0) circle (0.15);
				\draw [fill] (1.0,0.0) circle (0.15);
				\draw [fill=white] (3.0,0.0) circle (0.15);
				\draw [fill=white] (5.0,0.0) circle (0.15);
				\draw [fill=white] (-3.0,0.0) circle (0.15);
				\draw [fill=white] (-5.0,0.0) circle (0.15);
				\draw [fill] (-5.0,2.0) circle (0.15);
				\draw [fill] (-3.5,2.0) circle (0.15);
				\draw [fill=white] (-2.5,2.0) circle (0.15);
				\draw [fill] (-1.0,2.0) circle (0.15);
				\draw [fill] (-1.0,4.0) circle (0.15);
				\draw [fill] (1.0,2.0) circle (0.15);
				\draw [fill] (1.0,4.0) circle (0.15);
				\draw [fill=white] (3.5,2.0) circle (0.15);
				\draw [fill] (3.5,4.0) circle (0.15);
				\draw [fill] (-3.75,4.0) circle (0.15);
				\draw [fill=white] (-2.75,4.0) circle (0.15);
				\draw [fill] (-2.25,4.0) circle (0.15);
				\draw [fill=white] (-3.25,4.0) circle (0.15);
				\draw [fill] (2.5,2.0) circle (0.15);
				\draw [fill] (2.75,4.0) circle (0.15);
				\draw [fill=white] (2.25,4.0) circle (0.15);
				\draw [fill] (3.0,2.0) circle (0.15);
				\draw [fill=white] (4.0,2.0) circle (0.15);
				\draw [fill] (5.0,2.0) circle (0.15);
				\draw [fill] (5.0,4.0) circle (0.15);
				\draw [fill=white] (5.5,2.0) circle (0.15);
				\draw [fill] (4.5,2.0) circle (0.15);
				\draw [fill=white] (-7.0,0.0) circle (0.15);
				\draw (-5,5.2)node{Side $1$};
				\draw (3,5.2)node{Side $2$};
				\draw [dashed] (0,6)--(0,-1);
				\draw [dashed] (-2,6)--(-2,-1);
				\draw (-1,-0.5)node[below]{$p_{n/2}$};
				\draw (1,-0.5)node[below]{$p_{n/2+1}$};
				\draw (-3,-0.5)node[below]{$p_{n/2-1}$};
				\draw (-3,4.3)node[above]{$u$};
				\draw (3.95,2.3)node[above]{$v$};
				\end{tikzpicture}
			\end{center}
\caption{Illustration for the case when $u \in H_{n/2-1}$ and $v\notin H_{n/2+1}$.} \label{uv}
\end{figure}
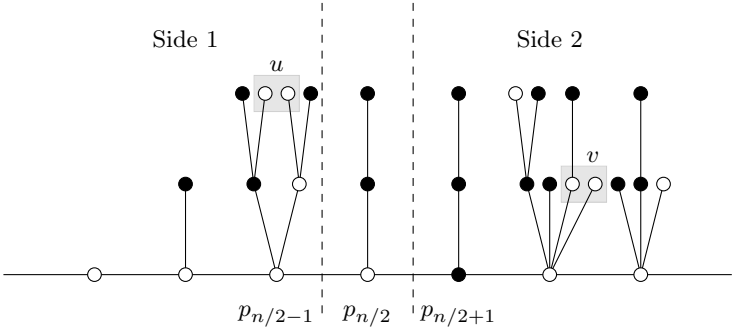			

Since $p_{n/2+1} \in R$, then $2\in r_m(v|R)$. In this case, $u$ is a leaf of an $S_4$ component in $[H_{n/2-1}]$ and $u\notin R$. If $v\notin R$, then by the maximality of $R(H_{n/2+1})$, there exists a component in $[H_{n/2+1}]$ with type 1 or type 12, and so $m_v(2)\geq 2$. However $m_u(2)\leq 1$, which leads to $rm(u|R)\ne rm(v|R)$.
			
\end{description}

In both constructions, we we prove that $R$ is an m-resolving set for $G$, and therefore $md(G)$ is finite.
	
\item[\textbf{$(1)\Rightarrow (2)$} ] Let $R$ be an m-resolving set of $G$ and suppose that $G-E(P)$ has a component $H \not\approx S_4$ with infinite multiset dimension. Let $v$ be the vertex in $H$ which is also in $P$. We will prove that there are two vertices $x$ and $y$ in $H$ with the same distance to $v$ and thus $r_m(x|R_H)=r_m(y|R_H)$.
			
Let $N_H(v)=\{v_1,v_2,\cdots,v_{deg(v)}\}$. If there exists a vertex $v_i$ in $N_H(v)$ with degree at least $4$ then there is at least $3$ leaves attached to $u$, and by Lemma \ref{twins}, there exists two vertices with the same representation. Now assume that $deg(v_i)\leq 3$ for all $i$. By the assumption of $H$, we have $deg(v)\geq2$. If $[H]$ contains an $S_4$ with two leaves other than $v$, then both are either in $R$ or not in $R$, and so the two leaves will have the same representation. Now the only case to consider is when the components of $[H]$ are a $P_2$, a $P_3$, or an $S_4$ with exactly one leaf (other than $v$) in $R$. If all of the components of $[H]$ has different types then we already established that $md(H)$ is finite. So there are components with the same type which means the neighbors of $v$ in those components will have the same representation.
			
Let $x$ and $y$ be the two vertices in $H$ with the same distance to $v$ with $r_m(x|R_H)=r_m(y|R_H)$. The path from $x$ or $y$ to any vertex $r$ in $R-R_H$ goes through $v$, and so $d(r,x)=d(r,y)$. Thus, $r_m(x|R)=r_m(y|R)$, a contradiction. We conclude that $H$ is either with finite multiset dimension or is an $S_4$.
\end{itemize}
\end{proof}

Note that if we let $P=p_0p_1\dots p_n$ to be any $2$-center path, not necessarily the minimum, then it is possible that the components of $G-E(P)$ satisfy (3) but not (1). One example is when $H_0=K_1$ and $H_1=P_3$, since $p_1$ has 3 leaves as neighbors.

If either $P$ is a path not containing an end-vertex; or $H_0$ and $H_n$ are either a $P_2$ or a $P_3$; or $[H_1]$ and $[H_{n-1}]$ only have at most $3$ components which are either a $P_2$, a $P_3$, or an $S_4$, with at most one $S_4$ and two $P_2$s, then Theorem \ref{Lobster} still hold. These assumptions are redundant for characterizing lobsters since they are equivalent with (3) in the theorem. However they could be used to characterise caterpillar with finite multiset dimension as stated in the following.

\begin{theorem}\label{Catterpillar}
Let $G$ be a caterpillar with $P$ its minimum $1$-center-path. $md(G)$ is finite if and only if every vertex in $P$ has at most $2$ neighbors in $G-P$.
\end{theorem}
	
In Theorem \ref{Catterpillar}, it is necessary for $P=p_0p_1\cdots p_n$ to be a minimum $1$-center path. An example to show its necessity is when $p_0$ is a leaf and $p_1$ have two neighbors in $G-P$.


Theorems \ref{Catterpillar} and \ref{Lobster} partially answered the question proposed in Problem \ref{chartree}. However, to characterize all trees with finite multiset dimension, we might have to use a different approach. We believe that applying our argument inductively to the size of the minimum center-path of a tree will be difficult to prove.

\begin{acknowledgement}
This research was partially supported by Penelitian Dasar Unggulan Perguruan Tinggi 2017-2019, funded by Indonesian Ministry of Research, Technology and Higher Education.
\end{acknowledgement}


\end{document}